\documentclass[12pt]{article}
\usepackage[utf8]{inputenc}
\usepackage{amsthm}
\usepackage{amsmath}
\usepackage{amsfonts}
\usepackage{epsfig}
\usepackage{txfonts}
\usepackage{float}
\newtheorem{theorem}{Theorem}
\newtheorem{lemma}[theorem]{Lemma}
\newtheorem{corollary}[theorem]{Corollary}
\newtheorem{observation}[theorem]{Observation}
\newcommand{\dom}{\gamma}
\newcommand{\wcol}{\text{wcol}}

\title{On distance $r$-dominating and $2r$-independent sets in sparse graphs}

\author{Zdeněk Dvořák\thanks{Charles University, Prague, Czech Republic.
E-mail: {\tt rakdver@iuuk.mff.cuni.cz}.  Supported by project 17-04611S (Ramsey-like aspects of graph
coloring) of Czech Science Foundation.}}
\date{}

\begin{document}
\maketitle

\begin{abstract}
Dvořák~\cite{apxdomin} gave a bound on the minimum size of a distance $r$ dominating set in the terms
of the maximum size of a distance $2r$ independent set and generalized coloring numbers, thus obtaining
a constant factor approximation algorithm for the parameters in any class of graphs with bounded expansion.
We improve and clarify this dependence using an LP-based argument inspired by the work of Bansal and Umboh~\cite{BANSAL201721}.
\end{abstract}

A set $X$ of vertices of a graph $G$ is \emph{dominating} if each vertex of $G$ either belongs to or has a neighbor in $X$,
and it is \emph{independent} if no two vertices of $X$ are adjacent.
The \emph{domination number} $\dom(G)$ of $G$ is the minimum size of a dominating set in $G$,
and the \emph{independence number} $\alpha(G)$ of $G$ is the maximum size of an independent set in $G$.
Determining either of these parameters in a general graph is NP-complete~\cite{Karp}.  Even approximating them
is hard. No polynomial-time algorithm approximating the domination number of an $n$-vertex graph within a factor better than $O(\log n)$
exists~\cite{raz1997sub}, unless $\text{P}=\text{NP}$.  Even worse,  for every $\varepsilon>0$, 
no polynomial-time algorithm approximating the independence number of an $n$-vertex graph within a factor better than $O(n^{1-\varepsilon})$
exists~\cite{haastad1999clique}, unless $\text{ZPP}=\text{NP}$.

Both parameters become more tractable in sparse graphs---they have a PTAS in planar graphs~\cite{baker1994approximation}
and other related graph classes, most generally in all graph classes with strongly sublinear separators~\cite{har2015approximation}.
To obtain constant-factor approximation, much weaker constraints suffice.  Lenzen and Wattenhofer~\cite{apxdomindeg}
proved that the domination number can be approximated within factor $a^2+3a+1$ on graphs with \emph{arboricity} at most $a$
(i.e., for graphs whose edge sets can be partitioned into at most $a$ forests).  This was improved to $3a$ by Bansal and Umboh~\cite{BANSAL201721}
using a simple LP-based argument, even under a weaker assumption.
\begin{theorem}[Bansal and Umboh~\cite{BANSAL201721}]\label{thm-domap}
For any positive integer $a$, if a graph $G$ has an orientation with maximum indegree at most $a$, then a dominating set in $G$ of size at most $3a\dom(G)$ can be found
in polynomial time.
\end{theorem}
Regarding independent sets, let us remark that if $G$ has an orientation with maximum indegree at most $a$, then its maximum average degree is at most $2a$. Consequently, $G$ has a proper coloring
using at most $2a+1$ colors, and one of the color classes gives an independent set of size at least $|V(G)|/(2a+1)$, which approximates the independence number within
the factor $2a+1$.

We consider distance generalizations of domination and independence number and the relationship between them.
A set $X\subseteq V(G)$ of vertices of a graph $G$ is \emph{$r$-dominating} if each vertex of $G$ is at distance at most $r$ from $X$.
For a vertex $v\in V(G)$, let $N_r[v]$ denote the set of vertices of $G$ at distance at most $r$ from $v$.
A set $Y\subseteq V(G)$ is \emph{$2r$-independent} if the distance between any two vertices of $Y$ is greater than $2r$,
or equivalently, if $|N_r[v]\cap Y|\le 1$ for each $v\in V(G)$.

Since each vertex $r$-dominates at most one vertex of a $2r$-independent set, we have $|Y|\le |X|$ for every $r$-dominating set $X$ and $2r$-independent set $Y$ in the graph $G$.
Hence, defining $\dom_r(G)$ as the minimum size of an $r$-dominating set in $G$ and $\alpha_{2r}(G)$
as the maximum size of a $2r$-independent set in $G$, we have the following inequality.
\begin{observation}\label{obs-rel}
For any positive integer $r$, every graph $G$ satisfies $$\alpha_{2r}(G)\le \dom_r(G).$$
\end{observation}
The relationship between $\dom_r$ and $\alpha_{2r}$ becomes clearer when we consider their LP relaxations.
Let
\begin{align*}
\dom^\star_r(G)&=\min \sum_{v\in V(G)} x_v\\
\text{subject to}&\\
\sum_{v\in N_r[u]} x_v&\ge 1&\text{for all $u\in V(G)$}\\
x_v&\ge 0&\text{for all $v\in V(G)$},
\end{align*}
and
\begin{align*}
\alpha^\star_{2r}(G)&=\max \sum_{u\in V(G)} y_u\\
\text{subject to}&\\
\sum_{u\in N_r[v]} y_u&\le 1&\text{for all $v\in V(G)$}\\
y_u&\ge 0&\text{for all $u\in V(G)$}.
\end{align*}
Since the programs defining the two parameters are dual, we obtain the following chain of inequalities.
\begin{observation}\label{obs-relfr}
For any positive integer $r$, every graph $G$ satisfies $$\alpha_{2r}(G)\le \alpha^\star_{2r}(G)=\dom^\star_r(G)\le \dom_r(G).$$
\end{observation}
Note that $\alpha^\star_{2r}(G)=\dom^\star_r(G)$ can be determined exactly in polynomial time by solving the linear programs that define them.

For any integer $r\ge 1$, the ratio $\dom_r(G)/\alpha_{2r}(G)$ can be arbitrarily large even for graphs of arboricity at most $3$,
the class of graphs studied by Lenzen and Wattenhofer~\cite{apxdomindeg}, as we will see below.  Dvořák~\cite{apxdomin} found a bound on this ratio
in terms of a stronger sparsity parameter.
Let $v_1$, $v_2$, \ldots, $v_n$ be an ordering of the vertices of a graph $G$.
A vertex $v_a$ is \emph{weakly $k$-accessible} from $v_b$ if $a\le b$ and there exists a path $v_a=v_{i_0}, v_{i_1}, \ldots, v_{i_{\ell}}=v_b$ of length $\ell\le k$ in $G$ such that
$a\le i_j$ for $0\le j\le \ell$. 
For a fixed ordering of $V(G)$, let $Q_k(v)$ denote the set of vertices that are weakly $k$-accessible from $v$
and let $q_k(v)=|Q_k(v)|$.  The \emph{weak $k$-coloring number} of the ordering is the maximum of $q_k(v)$ over $v\in V(G)$.
The weak $k$-coloring number $\wcol_k(G)$ of $G$ is the minimum of the weak $k$-coloring numbers
over all orderings of $V(G)$.

Weak coloring numbers were first defined by Kierstead and Yang~\cite{kierstead2003orderings}
as a distance generalization of degeneracy or ordinary coloring number---for any integer $d$,
a graph has weak $1$-coloring number at most $d+1$ if and only if $G$ is $d$-degenerate, i.e., each subgraph of $G$ has a vertex of degree at most $d$.
Weak $2$-coloring number is similarly related to another well studied graph parameter, \emph{arrangeability}~\cite{arr2, arr3, arr1}.
They also play an important role in the theory of graph classes with bounded expansion~\cite{nesbook}.
The starting point of our discourse is the following bound on the ratio $\dom_r(G)/\alpha_{2r}(G)$.
\begin{theorem}[Dvořák~\cite{apxdomin}]\label{thm-apxdomin}
For any positive integer $r$, every graph $G$ satisfies
$$\dom_r(G)\le \wcol_{2r}^2(G)\alpha_{2r}(G).$$
\end{theorem}
On the other hand, the ratio $\dom_r(G)/\alpha_{2r}(G)$ cannot be bounded by a function
of $\wcol_{2r-1}(G)$. Let us now give a construction showing this fact.
For a hypergraph $H$ and a positive integer $r$, let $H^{(r)}$ be the graph obtained as follows.  Let $H_1$ be
the incidence graph of $H$, i.e., the bipartite graph with parts $V(H)$ and $E(H)$ such that $v\in V(H)$ is adjacent to $e\in E(H)$
iff $v\in e$.  Let $H_2$ be the graph obtained from $H_1$ by subdividing each edge by $r-1$ vertices.  Finally, $H^{(r)}$
is obtained from $H_2$ by adding a new vertex $u$ joined by new paths of length $r$ to all vertices of $V(H_2)\setminus V(H)$.
The relevant properties of $H^{(r)}$ are given by the following lemma, which we prove in Section~\ref{sec:wca}.
\begin{lemma}\label{lemma-constr}
Let $H$ be a hypergraph of minimum degree $\delta\ge 1$ and let $r$ be a positive integer.
\begin{itemize}
\item $\wcol_{r-1}(H^{(r)})\le r^2-r+3$.
\item If each edge of $H$ has size most $t$, then $\wcol_{2r-1}(H^{(r)})\le r^2-r+t+2$.
\item If each two vertices of $H$ are incident with a common edge, then $\alpha_{2r}(H^{(r)})\le 2$.
\item If no $p$ edges of $H$ cover all vertices of $H$, then $\dom_r(H^{(r)})>p$.
\item $\dom_r^\star(G)=\alpha^\star_{2r}(G)\le |E(H)|/\delta+1$.
\end{itemize}
\end{lemma}

In particular, taking $H$ as the complete graph $K_n$, we have $\alpha_{2r}(K^{(r)}_n)\le 2$, $\dom_r(K^{(r)}_n)\ge n/2$,
and $\wcol_{2r-1}(K^{(r)}_n)\le r^2-r+4$, showing that $\dom_r/\alpha_{2r}$ cannot be bounded by a function of
$\wcol_{2r-1}$.  Furthermore, observe that the graph $K^{(r)}_n$ has arboricity at most $3$.

By Observation~\ref{obs-relfr}, $\alpha^\star_{2r}(G)=\dom^\star_r(G)$ approximates $\dom_r(G)$ and $\alpha_{2r}(G)$ within factor of $\wcol_{2r}^2(G)$.
Let us remark that Amiri et al.~\cite{amiri2017distributed} gave an improved approximation algorithm for $\dom_r$, within factor of $\wcol_{2r}(G)$.

In this note, we aim to clarify the relationship between distance domination number, distance independence number, and weak coloring numbers.
Firstly, generalizing the LP-based approach of Bansal and Umboh~\cite{BANSAL201721}, we show in Theorem~\ref{thm-dom} that $\dom_r(G)$ can be approximated within a factor expressed
in terms of $\wcol_r(G)$, by bounding the ratio $\dom_r(G)/\dom^\star_r(G)$; we also show that the ratio cannot be bounded
in terms of $\wcol_{r-1}(G)$.

Note that $\alpha^\star_{2r}(K^{(r)}_n)\ge n/2$, as shown by setting $y_v=1/2$ for $v\in V(K_n)$ and
$y_v=0$ for all other vertices $v$ of $K^{(r)}_n$; and thus the ratio $\alpha^\star_{2r}(G)/\alpha_{2r}(G)$ cannot be bounded
even in terms of $\wcol_{2r-1}(G)$.
To work around this issue, we consider a relaxed version of $2r$-independent set.  We say that a set $Y\subseteq V(G)$
is \emph{$(2r,b)$-independent} if $|N_r[v]\cap Y|\le b$ for all $v\in V(G)$.  Let $\alpha_{2r,b}(G)$ be the maximum
size of a $(2r,b)$-independent set in $G$.  Note that setting $y_v=1/b$ for $v\in Y$ and $y_v=0$ for $v\in V(G)\setminus Y$
gives a feasible solution to the program defining $\alpha^\star_{2r}$, and thus we have the following.
\begin{observation}
For any graph $G$ and positive integers $r$ and $b$,
$$\alpha_{2r}(G)\le \alpha_{2r,b}(G)\le b\alpha^\star_{2r}(G).$$
\end{observation}
We show in Theorem~\ref{thm-ind} that for $b=\wcol^2_r(G)$, the ratio $\alpha^\star_{2r}(G)/\alpha_{2r,b}(G)$ is at most $2$.

Finally, we link the results with Theorem~\ref{thm-apxdomin}.  In Lemma~\ref{lemma-makeindep}, we show that the ratio $\alpha_{2r,b}(G)/\alpha_{2r}(G)$ is
bounded in terms of $\wcol_{2r}(G)$.  Composing all the results, we obtain the following chain of inequalities.
\begin{theorem}\label{thm-summary}
For any graph $G$ and a positive integer $r$,
$$\frac{1}{\wcol^2_r(G)}\dom_r(G)\le \dom^\star_r(G)=\alpha^\star_{2r}(G)\le 2\alpha_{2r,\wcol^2_r(G)}(G)\le 4\wcol^2_r(G)\wcol_{2r}(G)\alpha_{2r}(G).$$
\end{theorem}
In particular, $\dom_r(G)\le 4\wcol^4_r(G)\wcol_{2r}(G)\alpha_{2r}(G)$, improving Theorem~\ref{thm-apxdomin} when $\wcol_{2r}(G)$ is large enough compared to $\wcol_r(G)$.
By Theorem~\ref{thm-summary}, $\dom^\star_r(G)=\alpha^\star_{2r}(G)$ approximates $\dom_r(G)$ and $\alpha_{2r}(G)$ up to factors $\wcol^2_r(G)$ and $4\wcol^2_r(G)\wcol_{2r}(G)$,
respectively.
Inspection of our arguments shows that they are constructive and give a polynomial-time algorithm to return an $r$-dominating set of size at most $w^2_r\dom^\star_r(G)$
and a $2r$-independent set of size at least $\frac{1}{4w^2_rw_{2r}}\alpha^\star_{2r}(G)$, assuming that orderings of vertices of $G$ with weak $r$-coloring number $w_r$
and weak $2r$-coloring number $w_{2r}$ are given.  See~\cite{apxdomin,grohe2014deciding} for a discussion of the complexity and algorithms to obtain such orderings.

\section{Weak coloring number and augmentations}\label{sec:wca}

Let us start by showing the properties of the graphs $H^{(r)}$ declared in Lemma~\ref{lemma-constr}.
Let us fix the notation as in the construction of the graph.  For a vertex $v\in V(H^{(r)})\setminus (V(H)\cup E(H)\cup \{u\})$,
let $R_v$ denote the set of vertices reachable from $v$ by paths in $H^{(r)}$ not containing vertices of $V(H)\cup E(H)\cup \{u\}$,
and for the unique $e\in R_v\cap E(H)$, let $R'_v=R_v\cup e$.
We have $|R_v|\le r^2-r+3$ and $|R'_v|\le r^2-r+|e|+2$.

\begin{proof}[Proof of Lemma~\ref{lemma-constr}]
Consider the ordering of the vertices of $H^{(r)}$ where $u$ is the smallest vertex, followed by vertices of $V(H)$ in
any order, vertices of $E(H)$ in any order, and finally all other vertices of $H^{(r)}$ in any order.
Clearly $Q_{r'}(u)=\{u\}$ for every $r'$.  If $v\in V(H)$, then $Q_{r'}(v)\subseteq \{u,v\}$ when $r'\le 2r-1$.
If $e\in E(H)$, then $Q_{r'}(e)=\{e\}$ when $r'\le r-1$ and $Q_{r'}(e)\subseteq e\cup \{e,u\}$ when $r'\le 2r-1$.
For any other vertex $v\in V(H^{(r)})$, we have $Q_{r'}(v)\subseteq R_v$ when $r'\le r$
and $Q_{r'}(v)\subseteq R'_v$ when $r'\le 2r$.  We conclude that $\wcol_{r-1}(H^{(r)})\le r^2-r+3$
and $\wcol_{2r-1}(H^{(r)})\le r^2-r+t+2$.

Note that any two vertices of $V(H^{(r)})\setminus V(H)$ are at distance at most $2r$ from one another, as shown by a path
through $u$.  If $u,v\in e$ for some $e\in E(H)$, then the distance between $u$ and $v$ in $H^{(r)}$ is $2r$, as shown
by a path through $e$.  Hence, if each two vertices of $H$ are contained in a common edge, then any $2r$-independent set
in $H^{(r)}$ has at most two vertices (one in $V(H)$ and one in $V(H^{(r)})\setminus V(H)$), and thus $\alpha_{2r}(H^{(r)})\le 2$.

No vertex of $V(H)$ is at distance at most $r$ from $u$ in $H^{(r)}$. If $v\in V(H)$, then no vertex of $V(H)\setminus\{v\}$ is
at distance at most $r$ from $v$ in $H^{(r)}$.  If $v=e\in E(H)$, or $v\in V(H^{(r)})\setminus (V(H)\cup E(H)\cup \{u\})$
and $e$ is the unique element of $R_v\cap E(H)$, then no vertices of $V(H)\setminus e$ are at distance at most $r$ from $u$.
Hence, for each $v\in V(H^{(r)})$, there exists $e_v\in E(H)$ such that all vertices of $V(H)$ at distance at most $r$ from $v$
in $H^{(r)}$ belong to $e_v$.  If $D$ is a dominating set in $H^{(r)}$, it follows that $\bigcup_{v\in D} e_v=V(H)$.
If no $p$ edges of $H$ cover all vertices of $H$, then we conclude that $\dom_r(H^{(r)})>p$.

Finally, setting $x_u=1$, $x_e=1/\delta$ for all $e\in E(H)$, and $x_v=0$ for all $v\in V(H^{(r)})$ gives a feasible solution
to the program defining $\dom_r^\star(G)$, which implies that $\dom_r^\star(G)\le |E(H)|/\delta+1$.
\end{proof}

In the rest of the paper, it is convenient to perform the arguments in terms of certain directed graphs rather than weak coloring numbers; this also makes
the connection to the result of Bansal and Umboh~\cite{BANSAL201721} more transparent.

For a positive integer $r$, an \emph{$r$-augmentation} $\widehat{G}$ of a graph $G$ is an orientation of a supergraph of $G$ with $V\bigl(\widehat{G}\bigr)=V(G)$
such that each edge $e\in E\bigl(\widehat{G}\bigr)$ is assigned \emph{length} $\rho(e)\in\{0,1,\ldots,r\}$ with the following properties:
\begin{itemize}
\item[(LOOP)] Each vertex of $\widehat{G}$ is incident with a loop of length $0$.
\item[(DIST)] For any non-negative integer $r'\le r$ and vertices $u,v\in V(G)$, the distance between $u$ and $v$ in $G$ is at most
$r'$ if and only if there exists a common inneighbor $x$ of $u$ and $v$ in $\widehat{G}$ and $\rho(xu)+\rho(xv)\le r'$.
\end{itemize}
Because of (LOOP), if $uv\in E\bigl(\widehat{G}\bigr)$, then (DIST) applied with $x=u$ shows that the distance between $u$ and $v$ in $G$ is at most $\rho(uv)$.
Essentially, we augment $G$ by adding (directed) edges $uv$ representing certain paths in $G$, with $\rho(uv)$ being the length of the
corresponding path between $u$ and $v$.  It follows that $\rho(uv)>0$ for all $uv\in E\bigl(\widehat{G}\bigr)$ with $u\neq v$,
and that if $x$ and $y$ are adjacent vertices of $G$, then at least one of the directed edges $xy$ or $yx$
appears in $\widehat{G}$ with length $1$.

For a non-negative integer $r'\le r$ and $v\in V(G)$, let $\deg_{r',\widehat{G}}(v)$ be the number of
inneighbors $u$ of $v$ with $\rho(uv)\le r'$, and let $\Delta_{r'}\bigl(\widehat{G}\bigr)$
be the maximum of $\deg_{r',\widehat{G}}(v)$ over all vertices of $G$.
Let us note a connection between weak coloring numbers and augmentations.
\begin{observation}\label{obs-wcol}
Consider any ordering of vertices of a graph $G$ and a non-negative integer $r$.  Let $\widehat{G}$ be the directed graph in which
$uv\in E\bigl(\widehat{G}\bigr)$ iff $u\in Q_r(v)$, and let $\rho(uv)$ be the minimum $r'$ such that $u\in Q_{r'}(v)$.  Then
$\widehat{G}$ is an $r$-augmentation of $G$ and $\Delta_{r'}\bigl(\widehat{G}\bigr)$ is equal to the
weak $r'$-coloring number of the ordering for any non-negative $r'\le r$.
\end{observation}
\begin{proof}
Note that $\Delta_{r'}\bigl(\widehat{G}\bigr)$ is equal to the weak $r'$-coloring number of the ordering for every $r'\le r$ by the choice of $\rho$, and thus it suffices
to argue that $\widehat{G}$ is an $r$-augmentation.

Since $v\in Q_0(v)$ for all $v\in V(G)$, (LOOP) is satisfied by $\widehat{G}$.  If vertices $u,v\in V(G)$ have a common inneighbor $x$ in $\widehat{G}$,
then their distance in $G$ is at most the sum of distances from $x$ to $u$ and $v$, which is at most $\rho(xu)+\rho(xv)$.
Conversely, suppose that the distance between $u$ and $v$ is $r'\le r$.  Let $P$ be a path of length $r'$ from $u$ to $v$,
and let $x$ be the smallest vertex of $P$ in the considered ordering.  Then $x\in Q_{r_1}(u)\cap Q_{r_2}(v)$,
where $r_1$ and $r_2$ are the lengths of the subpaths of $P$ from $x$ to $u$ and $v$, and $r'=r_1+r_2\ge \rho(xu)+\rho(xv)$.
We conclude that (DIST) holds.
\end{proof}
Note that it is possible to obtain $r$-augmentations in other ways, e.g., using the transitive fraternal augmentation
procedure of Nešetřil and Ossona de Mendez~\cite{grad1}.

\section{Domination}

We are now ready to give the approximation argument for $\dom_r$.

\begin{theorem}\label{thm-dom}
Let $r$ be a positive integer. If $\widehat{G}$ is an $r$-augmentation of a graph $G$, then
$$\dom_r(G)/\dom^\star_r(G)\le (\Delta_{r-1}\bigl(\widehat{G}\bigr)+1)\Delta_r\bigl(\widehat{G}\bigr)-\Delta_{r-1}\bigl(\widehat{G}\bigr).$$
\end{theorem}
\begin{proof}
Let $a=(\Delta_{r-1}\bigl(\widehat{G}\bigr)+1)\Delta_r\bigl(\widehat{G}\bigr)-\Delta_{r-1}\bigl(\widehat{G}\bigr)$.
Consider any optimal solution to the linear program defining $\dom^\star_r(G)$, and let
$X_0$ be the set of vertices $v\in V(G)$ such that $x_v\ge 1/a$ in this solution.
Let $v_1$, \ldots, $v_n$ be any ordering of vertices of $G$.  For $i=1, \ldots, n$,
if a vertex at distance at most $r$ from $v_i$ belongs to $X_{i-1}$, then let $X_i=X_{i-1}$;
otherwise, $X_i$ is obtained from $X_{i-1}$ by adding all inneighbors $x$ of $v_i$ such that $\rho(xv_i)\le r-1$.

Clearly, $X_n$ is an $r$-dominating set of $G$; hence, it suffices to bound its size.
We have
$$|X_0|\le a\sum_{u\in X_0} x_u.$$
To bound $|X_n\setminus X_0|$, we perform a charge redistribution
argument.  Vertices start with zero charge. For $i=1,\ldots, n$, if $X_i\neq X_{i-1}$, then we increase by $x_u$ the charge of each vertex
$u\in N_r[v_i]$ such that $uv_i\not\in E\bigl(\widehat{G}\bigr)$.
Let $\delta_i$ denote the total amount of charge added in this step.
Observe that since $X_i\neq X_{i-1}$,
none of vertices in $N_r[v_i]$ belongs to $X_{i-1}$.
In particular, no inneighbor of $v_i$ belongs to $X_0$, and thus $\sum_{uv_i\in E(\widehat{G})} x_u\le \Delta_r\bigl(\widehat{G}\bigr)/a$.
Since we are considering a solution to the linear program defining $\dom^\star_r(G)$, we
have the following bound on the charge increase.
$$\delta_i=\sum_{u\in N_r[v_i],uv_i\not\in E\bigl(\widehat{G}\bigr)} x_u\ge 1-\sum_{uv_i\in E\bigl(\widehat{G}\bigr)} x_u\ge 1-\Delta_r\bigl(\widehat{G}\bigr)/a=\Delta_{r-1}\bigl(\widehat{G}\bigr)(\Delta_r\bigl(\widehat{G}\bigr)-1)/a.$$
Consequently,
$$|X_i\setminus X_{i-1}|\le \Delta_{r-1}\bigl(\widehat{G}\bigr)\le \delta_i\frac{a}{\Delta_r\bigl(\widehat{G}\bigr)-1},$$
and letting $\delta=\sum_{i=1}^n \delta_i$ be the total amount of charge created, we have
$$|X_n\setminus X_0|\le \delta\frac{a}{\Delta_r\bigl(\widehat{G}\bigr)-1}.$$
On the other hand, by (DIST), when $u\in N_r[v_i]$ and $uv_i\not\in E\bigl(\widehat{G}\bigr)$, then $u$ and $v_i$ have a common inneighbor $x$
and $r\ge \rho(xu)+\rho(xv_i)\ge 1+\rho(xv_i)$.  Consequently, whenever the charge of $u$ is increased, some inneighbor of $u$ (distinct from $u$)
is added to the $r$-dominating set, and thus the final charge of $u$ is at most $(\Delta_r\bigl(\widehat{G}\bigr)-1)x_u$.  Furthermore, as we observed before, charge is only added
to vertices not belonging to $X_0$.  Summing over all vertices of $G$, we obtain
$$\delta\le (\Delta_r\bigl(\widehat{G})-1\bigr)\sum_{u\not\in X_0} x_u.$$
Combining these bounds, we obtain
$$|X_n\setminus X_0|\le a\sum_{u\not\in X_0}x_u,$$
and thus
$$\dom_r(G)\le |X_0|+|X_n\setminus X_0|\le a\sum_{u\in V(G)} x_u=a\dom_r^\star(G),$$
as required.
\end{proof}

Note that $(x+1)y-x=y^2-(y-x)(y-1)$, and thus if $1\le x\le y$, then
$(x+1)y\le y^2$.  Hence,
$(\Delta_{r-1}\bigl(\widehat{G}\bigr)+1)\Delta_r\bigl(\widehat{G}\bigr)-\Delta_{r-1}\bigl(\widehat{G}\bigr)\le \Delta^2_r\bigl(\widehat{G}\bigr)$.
By Observation~\ref{obs-wcol}, Theorem~\ref{thm-dom} has the following consequence.
\begin{corollary}\label{cor-domincol}
For any positive integer $r$ and a graph $G$,
$$\dom_r(G)/\dom^\star_r(G)\le \wcol^2_r(G).$$
\end{corollary}
On the other hand, the ratio cannot be bounded in terms of $\wcol_{r-1}$, as shown by the following example.
Let $n$ be an odd integer and let
$H$ be the hypergraph with vertex sets consisting of all subsets of $\{1,\ldots,n\}$ of size $(n+1)/2$, with
edges $e_1$, \ldots, $e_n$ such that for $1\le i\le n$, the edge $e_i$ consists of the sets in $V(H)$ that contain $i$.
For any $I\subseteq \{1,\ldots,n\}$ of size $(n-1)/2$, the vertex $\{1,\ldots,n\}\setminus I$ is not incident with
any of the edges $e_i$ for $i\in I$; hence, Lemma~\ref{lemma-constr} implies $\dom_r(H^{(r)})\ge (n+1)/2$.
Each vertex of $H$ is incident with $(n+1)/2$ edges and $|E(H)|=n$, and thus $\dom_r^\star(H^{(r)})\le \frac{n}{(n+1)/2}+1\le 3$.
Also, Lemma~\ref{lemma-constr} implies $\wcol_{r-1}(H^{(r)})\le r^2-r+3$.

Corollary~\ref{cor-domincol} implies that $\dom_r(G)$ can be approximated in polynomial time within factor of $\wcol^2_r(G)$.
This bound can be improved in the special case $r=1$.
If $G$ has an orientation with maximum indegree at most $d$, then giving each edge length $1$
and adding loops of length $0$ on all vertices results in a $1$-augmentation $\widehat{G}$ with $\Delta_1\bigl(\widehat{G}\bigr)\le d+1$
and $\Delta_0\bigl(\widehat{G}\bigr)=1$.  Hence, we have the following.
\begin{corollary}
If a graph $G$ has an orientation with maximum indegree at most $d$, then $\dom(G)$ can be approximated in polynomial time
within factor of $2d+1$.
\end{corollary}
Note that Bansal and Umboh~\cite{BANSAL201721} give the approximation factor as $3d$, however this is just because
we are slightly more careful in the analysis---their algorithm is exactly the same as the one of Theorem~\ref{thm-dom} in the case $r=1$.

\section{Independence}

To prove a bound on the ratio $\alpha^\star_{2r}(G)/\alpha_{2r,b}(G)$,
we use a result of Parekh and Pritchard~\cite{Parekh2015} on generalized hypergraph matching.
Let $H$ be a hypergraph and let $b$ be a positive integer.
A \emph{$b$-matching} in $H$ is a set $M$ of edges of $H$ such that each vertex of $H$ is incident with at most
$b$ edges of $M$.  Let $\mu_b(H)$ denote the maximum size of a $b$-matching in $H$.
Let $\mu^\star_b(H)$ be the fractional relaxation of this parameter, defined as
\begin{align*}
\mu^\star_b(H)&=\max \sum_{e\in E(H)} m_e\\
\text{subject to}&\\
\sum_{e\ni v} m_e&\le b&\text{for all $v\in V(H)$}\\
0\le m_e&\le 1&\text{for all $e\in E(H)$}
\end{align*}
Clearly, $\mu^\star_b(H)\ge \mu_b(H)$.  Conversely, we have the following.

\begin{theorem}[Parekh and Pritchard~\cite{Parekh2015}]\label{thm-magap}
If $H$ is a hypergraph with all edges of size at most $k$ and $b$ is a positive integer,
then
$$\mu_b(H)\ge \frac{k}{k^2-k+1}\mu_b^\star(H).$$
Furthermore, a $b$-matching of size at least $\frac{k}{k^2-k+1}\mu_b^\star(H)$ can be found in polynomial time.
\end{theorem}

\begin{corollary}\label{cor-maga}
If $H$ is a hypergraph with all edges of size at most $k\ge 1$, then
then $\mu_k(H)\ge \mu_1^\star(H)/2$.
\end{corollary}
\begin{proof}
Consider an optimal solution to the linear program defining $\mu_1^\star(H)$.
Let $M_1=\{e\in E(H):m_e\ge 1/k\}$ and $s_1=\sum_{e\in M_1} m_e$.
Clearly, $M_1$ is a $k$-matching in $H$, and thus $\mu_k(H)\ge |M_1|\ge s_1$.
If $s_1\ge \mu_1^\star(H)/2$, the desired bound on $\mu_k(H)$ follows, and thus
assume that $s_1<\mu_1^\star(H)/2$.

Let $m'_e=km_e$ for all $e\in E(H)$ such that $m_e\le 1/k$ and $m'_e=0$ for all other
$e\in E(H)$.  This gives a feasible solution to the program defining $\mu^\star_k(H)$, and thus
$$\mu^\star_k(H)\ge k(\mu_1^\star(H)-s_1)>k\mu_1^\star(H)/2.$$
By Theorem~\ref{thm-magap}, we have
$$\mu_k(H)>\frac{k}{k^2-k+1}k\mu_1^\star(H)/2\ge \mu_1^\star(H)/2,$$
as required.
\end{proof}

We use this result to find sets intersecting outneighborhoods in $r$-augmentations only in a bounded
number of vertices.

\begin{lemma}\label{lemma-out}
Let $G$ be a graph, let $\widehat{G}$ be an $r$-augmentation of $G$, and let $k=\Delta_r\bigl(\widehat{G}\bigr)$.
There exists a set $Y\subseteq V(G)$ such that
each vertex of $\widehat{G}$ has at most $k$ outneighbors in $Y$ and
$$|Y|\ge \alpha_{2r}^\star(G)/2.$$
\end{lemma}
\begin{proof}
For a vertex $u\in V(G)$, let $e_u$ be the set of inneighbors of $u$ in $\widehat{G}$.
Let $H$ be the hypergraph with vertex set $V(G)$ and edge set $\{e_u:u\in V(G)\}$;
each edge of $H$ has size at most $k$.  Note that $M$ is a $k$-matching in
$H$ if and only if each vertex of $\widehat{G}$ has at most $k$ outneighbors
in $Y=\{u:e_u\in M\}$.  Hence, it suffices to prove that $\mu_k(H)\ge \alpha_{2r}^\star(G)/2$.

Consider an optimal solution to the linear program defining $\alpha_{2r}^\star(G)$,
and for every $u\in V(G)$, let $m_{e_u}=y_u$.  For each $v\in V(H)$, we have
$$\sum_{e_u\ni v} m_{e_u}=\sum_{vu\in V\bigl(\widehat{G}\bigr)} y_u\le \sum_{u\in N_r[v]} y_u\le 1,$$
and thus this gives a feasible solution to the program defining $\mu^\star_1(H)$.
We conclude that $\mu^\star_1(H)\ge \alpha_{2r}^\star(G)$, and thus the claim follows
from Corollary~\ref{cor-maga}.
\end{proof}

We are now ready to show existence of large $(2r,b)$-independent sets.
\begin{theorem}\label{thm-ind}
Let $G$ be a graph, let $\widehat{G}$ be an $r$-augmentation of $G$, and
let $b=(\Delta_{r-1}\bigl(\widehat{G}\bigr)+1)\Delta_r\bigl(\widehat{G}\bigr)-\Delta_{r-1}\bigl(\widehat{G}\bigr)$.
We have $$\alpha_{2r,b}\ge \alpha_{2r}^\star(G)/2.$$
\end{theorem}
\begin{proof}
Let $k=\Delta_r\bigl(\widehat{G}\bigr)$, and let $Y$ be the set obtained by applying Lemma~\ref{lemma-out},
such that every vertex of $\widehat{G}$ has at most $k$ outneighbors belonging to $Y$.
For any $v\in V(G)$ and $y\in N_r[v]\cap Y$, (DIST) implies that either $y$ is an inneighbor of $v$ in $\widehat{G}$
and $\rho(yv)=r$, or $y$ and $v$ have a common inneighbor $x$ with $\rho(xv)\le r-1$.  Hence,
we have
\begin{align*}
|N_r[v]\cap Y|&\le (\deg_{r,\widehat{G}}(v)-\deg_{r-1,\widehat{G}}(v))+\sum_{xv\in E\bigl(\widehat{G}\bigr),\rho(xv)\le r-1} |\{y\in Y: xy\in E\bigl(\widehat{G}\bigr)|\\
&\le (\deg_{r,\widehat{G}}(v)-\deg_{r-1,\widehat{G}}(v))+\deg_{r-1,\widehat{G}}(v))\cdot k\\
&=\deg_{r,\widehat{G}}(v)+(k-1)\deg_{r-1,\widehat{G}}(v)\\
&\le \Delta_r\bigl(\widehat{G}\bigr)+(k-1)\Delta_{r-1}\bigl(\widehat{G}\bigr)=b,
\end{align*}
and thus $Y$ is a $(2r,b)$-independent set in $G$.
Consequently,
$$\alpha_{2r,b}(G)\ge |Y|\ge \alpha_{2r}^\star(G)/2,$$
as required.
\end{proof}

The following lemma clarifies the relationship with Theorem~\ref{thm-apxdomin}.
\begin{lemma}\label{lemma-makeindep}
If $\widehat{G}$ is a $2r$-augmentation of $G$,
then for every positive integer $b$,
$$\alpha_{2r}(G)\ge \frac{1}{2b\Delta_{2r}\bigl(\widehat{G}\bigr)}\alpha_{2r,b}(G).$$
\end{lemma}
\begin{proof}
Let $Y$ be a $(2r,b)$-independent set in $G$ of size $\alpha_{2r,b}(G)$.  Let $G_1$ be the graph with vertex set $Y$ and distinct vertices $y_1, y_2\in Y$
adjacent iff their distance in $G$ is at most $2r$.  Orient the edges of $G_1$ as follows: if $v$ is an inneighbor of $y_1$ in $\widehat{G}$
and $y_2\in N_r[v]$, then direct the edge from $y_2$ to $y_1$.
Since $Y$ is $(2r,b)$-independent, we have $|N_r[v]\cap (Y\setminus \{y_1\})|\le b$ for each inneighbor $v$ of $y_1$,
and the inequality is strict when $v=y_1$.  Hence, the maximum indegree of $G_1$ is less than $d=b\Delta_{2r}\bigl(\widehat{G}\bigr)$.
Furthermore, all edges of $G_1$ are directed in at least one direction by (DIST).  
Consequently, each subgraph $F$ of $G_1$ has less than $d|V(F)|$ edges,
and thus $G_1$ is $(2d-1)$-degenerate.  Consequently, $\chi(G_1)\le 2d$, and thus
$G_1$ contains an independent set $Y_1$ of size at least $\frac{|Y|}{2d}=\frac{\alpha_{2r,b}(G)}{2b\Delta_{2r}(\widehat{G})}$.
Observe that $Y_1$ is a $2r$-independent in $G$, which gives the required lower bound on $\alpha_{2r}(G)$.
\end{proof}

Composing Theorems~\ref{thm-dom} and \ref{thm-ind} with Lemma~\ref{lemma-makeindep}, and using
Observation~\ref{obs-wcol}, we obtain the following inequalities, implying Theorem~\ref{thm-summary}.
\begin{corollary}\label{cor-summary}
Let $G$ be a graph, let $\widehat{G_1}$ be an $r$-augmentation of $G$
and let $\widehat{G_2}$ be a $2r$-augmentation of $G$.
Let $b=(\Delta_{r-1}\bigl(\widehat{G_1}\bigr)+1)\Delta_r\bigl(\widehat{G_1}\bigr)-\Delta_{r-1}\bigl(\widehat{G_1}\bigr)$.
Then
$$\frac{1}{b}\dom_r(G)\le \dom^\star_r(G)=\alpha^\star_{2r}(G)\le 2\alpha_{2r,b}(G)\le 4b\Delta_{2r}\bigl(\widehat{G_2}\bigr)\alpha_{2r}(G).$$
In particular,
$$\frac{1}{\wcol^2_r(G)}\dom_r(G)\le \dom^\star_r(G)=\alpha^\star_{2r}(G)\le 2\alpha_{2r,\wcol^2_r(G)}(G)\le 4\wcol^2_r(G)\wcol_{2r}(G)\alpha_{2r}(G),$$
and if $G$ has an orientation with maximum indegree at most $d$, then
$$\frac{1}{2d+1}\dom(G)\le \dom^\star_1(G)=\alpha^\star_2(G)\le 2\alpha_{2,2d+1}(G)\le 4(2d+1)\wcol_2(G)\alpha_2(G).$$
\end{corollary}

\bibliographystyle{siam}
\bibliography{apxdomin}

\end{document}